\documentclass[a4paper,12pt,reqno]{amsart}
\usepackage[T1]{fontenc}
\usepackage[utf8]{inputenc}
\usepackage{bera}
\usepackage{amssymb,dsfont,mathrsfs,mathabx}
\usepackage[margin=1in]{geometry}
\usepackage{enumitem}
\usepackage{pdflscape,afterpage}
\usepackage[bookmarksdepth=2]{hyperref}
\usepackage{xcolor,graphicx}

\numberwithin{equation}{section}

\newtheorem{theorem}{Theorem}[section]
\newtheorem{lemma}[theorem]{Lemma}
\newtheorem{corollary}[theorem]{Corollary}
\newtheorem{proposition}[theorem]{Proposition}

\theoremstyle{definition}
\newtheorem{definition}[theorem]{Definition}

\newtheorem{remark}[theorem]{Remark}

\theoremstyle{notation}

\DeclareMathOperator*{\esssup}{ess\,sup}
\DeclareMathOperator*{\essinf}{ess\,inf}

\newcommand{\fourier}{\mathscr{F}}

\newcommand{\R}{\mathds{R}}
\newcommand{\sph}{\mathds{S}}

\newcommand{\N}{\mathds{N}}

\renewcommand{\le}{\leqslant}
\renewcommand{\ge}{\geqslant}

\newenvironment{PDE}
	{ \left \{
	\begin{array}{r@{ \ }l @{\quad \: \;} l}
	}
	{
	\end{array} \right . 
	}

\newcommand{\jj}{{\mathcal{K}}}
\definecolor{RED}{rgb}{1,0,0}

\begin{document}

\date{\today}

\title[The nonlocal Harnack inequality for antisymmetric solutions]{The nonlocal Harnack inequality \\ for antisymmetric solutions: \\
an approach via Bochner's relation and harmonic analysis}

\author[S. Dipierro, M.~Kwa\'snicki, J. Thompson, E. Valdinoci]{Serena Dipierro,
Mateusz Kwa\'snicki, Jack Thompson, Enrico Valdinoci}
\address{Serena Dipierro: Department of Mathematics and Statistics,
The University of Western Australia, 35 Stirling Highway, Crawley, Perth, WA 6009, Australia.}
\email{serena.dipierro@uwa.edu.au}
\address{Mateusz Kwa\'snicki: Department of Pure Mathematics, Wroc\l{}aw University of Science and Technology, Wyb.\@ Wyspia\'nskiego 27, 50-370 Wroc\l{}aw, Poland.}
\email{mateusz.kwasnicki@pwr.edu.pl}
\address{Jack Thompson: Department of Mathematics and Statistics,
The University of Western Australia, 35 Stirling Highway, Crawley, Perth, WA 6009, Australia.}
\email{jack.thompson@research.uwa.edu.au}
\address{Enrico Valdinoci: Department of Mathematics and Statistics, The University of Western Australia, 35 Stirling Highway, Crawley, Perth, WA 6009, Australia.}
\email{enrico.valdinoci@uwa.edu.au}

\begin{abstract} We revisit a Harnack inequality for antisymmetric functions
that has been recently established for the fractional Laplacian and we extend it to
more general nonlocal elliptic operators.

The new approach to deal with these problems that we propose in this paper
leverages Bochner's relation, allowing one to
relate a one-dimensional Fourier transform of an odd function with a three-dimensional Fourier transform of a radial function.

In this way, Harnack inequalities for odd functions, which are essentially 
Harnack inequalities of boundary type, are reduced to interior Harnack inequalities.
\end{abstract}

\maketitle

%
%

\section{Introduction}

One of the characterizing properties of classical harmonic functions is their ``rigidity'': in spite of the fact that harmonic functions may exhibit, in general, different patterns, a common feature is that if, at a given point, a harmonic function ``bends up'' in a certain direction, then necessarily it has to ``bend down'' in another direction. This observation is typically formalised by the so-called maximum principle. Furthermore, the maximum principle is, in turn, quantified by the Harnack inequality~\cite{HAR} which asserts that the values of a positive harmonic function in a given ball are comparable (see~\cite{MK} for a thorough introduction to the topic).

The growing interest recently surged in 
the study of nonlocal and fractional equations, especially in relation to fractional powers of the Laplace operator,
has stimulated an intense investigation on the possible validity of Harnack-type inequalities in a nonlocal framework.
Several versions of the Harnack inequality have been obtained for the fractional Laplacian as well as for more general nonlocal operators, see~\cite{BASS02, BOG03, SON04, BASS05, SIL06, CS07, STI10, DICA14, FER15}.\medskip

A striking difference between the classical and the fractional settings is that in the former
a sign assumption on the solution is taken only in the region of interest, while in the
latter such an assumption is known to be insufficient. Indeed, counterexamples to the validity of the nonlocal
Harnack inequality in the absence of a global sign assumptions have been put forth in~\cite{MK7} (see also~\cite[Theorem~3.3.1]{BUC} and~\cite[Theorem~2.2]{MK11}). In fact,
these counterexamples are just particular cases
of the significant effect that the faraway oscillations may produce
on the local patterns of solutions to nonlocal equations,
see~\cite{DSV, DSV2}
(in case of solutions which change sign, the nonlocal framework however allows
for a Harnack inequality with a suitable integral remainder, see~\cite[Theorem~2.3]{MK11}).\medskip

A rather prototypical situation in which the sign condition is violated is that of odd antisymmetric functions. 
Specifically, we consider here the case in which
antisymmetry occurs with respect to the first variable
(or, more precisely,
with respect to a hyperplane, which can be always assumed, without loss of generality, to be~$ \{x_1 = 0\}$; a detailed discussion will be presented in Section~\ref{sec:bochner}).

In the nonlocal world, 
antisymmetry situations with respect to a certain variable
frequently occur, especially when dealing with moving planes and reflection methods. On these occasions, the lack of sign assumption needs to be replaced by bespoke maximum principles which carefully take into account the additional antisymmetrical structure of the problem under consideration, see~\cite{FALL15, SV16, CO20, FEL20, FEL21, TR24}.
In particular, in the antisymmetric setting, the following
nonlocal Harnack inequality has been established in~\cite{dtv}:

\begin{theorem}
\label{thm:odd:harnack}
Suppose that $\Omega$ is an open set of~$\R^n$,
and let~$u$ be a function on $\R^n$ such that $u$ is $C^{2s + \delta}$ in $\Omega$ and antisymmetric with respect to the hyperplane~$\{x_1=0\}$,
with
\begin{equation*}
\int_{\R^n}\frac{|x_1|\,|  u(x)|}{(1 + |x|)^{n + 2 +2s}}\, dx<+\infty
\end{equation*}
and~$x_1 u(x) \ge 0$ for every $x \in \R^n$.

Let $c$ be a bounded Borel function on $\Omega$. If
\[
 (-\Delta)^s u + c u = 0
\]
in $\Omega_+ = \{x \in \Omega : x_1 > 0\}$, then for every compact subset $K$ of $\Omega$ we have
\begin{equation}
\label{eq:odd:harnack}
 \sup_{x \in K_0} \frac{u(x)}{x_1} \le C(K, \Omega, \|c\|_{L^\infty(\Omega_+)}) \inf_{x \in K_0} \frac{u(x)}{x_1} ,
\end{equation}
where $K_0 = \{x \in K : x_1 \ne 0\}$.
\end{theorem}

In this article, we revisit the nonlocal Harnack inequality under a different perspective,
providing a new, shorter proof of Theorem~\ref{thm:odd:harnack}. This proof relies
on the so-called Bochner's relation (or Hecke--Bochner identity), which links a one-dimensional Fourier transform of an odd function with a three-dimensional Fourier transform of a radial function: roughly speaking, for any function~$ f$ on the positive reals (under a natural integrability condition), if~$F(|Z|)$ denotes the \mbox{3-D} Fourier transform of the 3-D radial function~$f(|X|)$, then the 1-D Fourier transform of the \mbox{1-D} function~$x f(|x|)$ is equal to~$i z F(|z|)$. The complete statement of Bochner's relation is actually even more general (it works in higher dimensions too, and the variable~$x$ can be replaced by a homogeneous harmonic polynomial), see~\cite[page~72]{stein}.

The gist of the argument that we present here is that
Bochner's relation extends to the case of antisymmetric functions dealt with in
Theorem~\ref{thm:odd:harnack}. More specifically,
the Fourier transform of an \(n\)-dimensional antisymmetric function~$ x_1 f(|x_1|, x')$ is equal to~$i z_1 F(|z_1|, z')$, where~$F (|Z|, z')$ is the $(n + 2)$-dimensional Fourier transform of~$f(|X|, x')$, where~$X$ and~$ Z$ are 3-dimensional vectors, $x_1$ and~$z_1$ are real numbers, and~$x'$ and~$z'$ are~$(n - 1)$-dimensional vectors.

Using the above property, one finds that the $n$-dimensional fractional Laplacian applied to~$x_1 f(|x_1|, x')$ is equal to~$x_1 g(|x_1|, x')$ if and only if the~$(n + 2)$-dimensional fractional Laplacian applied to~$f(|X|, x')$ is equal to~$ g(|X|, x')$ (in fact, this is true for an arbitrary function of the Laplacian, not just for its fractional powers).

This, in turn, enables one to apply the usual Harnack inequality in dimension~$n + 2$ to deduce a Harnack inequality (or, in a sense, a ``boundary'' Harnack inequality)
in dimension~$ n$ for antisymmetric functions. In this framework,
Theorem~\ref{thm:odd:harnack} follows from a more classical result for positive solutions in dimension~$n + 2$.\medskip

For other applications of
Bochner's relation to antisymmetric functions, see~\cite[Theorem~1]{DYDA},
\cite[Theorem~1.5 and Proposition~3.1]{KR16} and~\cite{F-F}.
Generally speaking, Fourier analysis is a powerful tool in the study of partial differential equations and nonlocal equations. Its applications encompass the theory of pseudodifferential operators, extension methods, and various techniques for analyzing regularity, decay, and the propagation of singularities.\medskip

Besides its intrinsic elegance and the conciseness of the techniques involved, one of the advantages of the methods leveraging Bochner's relation consists of the broad versatility
of the arguments employed. In particular, the methodology employed allows us to
extend the previous result in Theorem~\ref{thm:odd:harnack} to a more general class of nonlocal operators. To this end, 
given a positive integer $n$, we consider an operator $L_n$, acting on an appropriate class of functions $u:\R^n\to \R$, and taking the form
\begin{equation}
 L_n u(x) = \int_{\R^n \setminus \{0\}} \bigl(u(x) - u(x + y) - \nabla u(x) \cdot y \chi_{B_1}(y)\bigr)\, \jj_n(|y|) dy ,
\label{5XOlUbY7}
\end{equation}%
where the kernel $\jj_n$ satisfies the usual integrability condition
\[
 \int_{\R^n \setminus \{0\}} \min\{1, |x|^2\} \, \jj_n(|x|) dx < +\infty .
\]
To proceed, we will need the following additional assumption:
\begin{equation}
\label{radialassumption}
 \text{$\jj_n$ is differentiable.}
\end{equation}
This allows us to introduce the operator $L_{n + 2}$, defined by the same formula with kernel
\begin{equation}
\label{nplus2}
 \jj_{n + 2}(r) = -\frac{\jj_n'(r)}{2 \pi r} \, .
\end{equation}
Though we will not rely on this additional assumption in the present paper, we recall that
in this framework a natural hypothesis is also to suppose that~$\jj_n$ is nonincreasing on $(0, \infty)$, so to guarantee that~$\jj_{n + 2}$ is nonnegative.

We establish that:

\begin{theorem}\label{THM5.55}
Let~$L_n$ and $L_{n + 2}$ be given by~\eqref{5XOlUbY7}, and suppose that conditions~\eqref{radialassumption} and~\eqref{nplus2} are satisfied. Suppose that $\Omega$ is an open set of~$\R^n$, and let $u$ be a function on~$\R^n$ such that $u\in C^{2}(\Omega)$, $x_1^{-1} u \in L^\infty(\R^n)$, and $u$ is antisymmetric with respect to the hyperplane~$\{x_1=0\}$,
with~$x_1u(x) \geqslant 0$ for all~$x\in \R^n$.

Let $c$ be a bounded Borel function on $\Omega$. If
\[
L_n u + c u = 0
\]
in $\Omega_+ = \{x \in \Omega : x_1 > 0\}$ and the operator $L_{n+2}$ satisfies the Harnack inequality, then for every compact subset $K$ of $\Omega$ we have
\begin{equation}
\label{eq:odd:harnack-general}
 \sup_{x \in K_0} \frac{u(x)}{x_1} \le C(K, \Omega, \|c\|_{L^\infty(\Omega_+)}) \inf_{x \in K_0} \frac{u(x)}{x_1} ,
\end{equation}
where $K_0 = \{x \in K : x_1 \ne 0\}$ and $C(K, \Omega, M)$ is the same as the constant in the Harnack inequality for $L_{n + 2}$.
\end{theorem}
For the precise definition of an operator satisfying the Harnack inequality, see Definition~\ref{mTbTJpHm}. The Harnack inequality (with a scale-invariant constant) is known to hold for a large class of operators $L_{n + 2}$, at least when $c = 0$; we refer to Remark~\ref{rem:harnack} in Section~\ref{OJSDLN09ryihfjg9reSTRpsdlmvIJJ} for further discussion.

The rest of the paper is organized as follows. In Section~\ref{SE:2}, we give some heuristic comments about the idea of reducing boundary Harnack inequalities to interior ones, to motivate the main strategy adopted in this paper, while, in Section~\ref{sec:bochner},
we recall Bochner's relation and frame it into a setting convenient for our purposes.

The techniques thus introduced are then applied to the case of the fractional Laplace operator in Section~\ref{sec:flap}, leading to the proof of Theorem~\ref{thm:odd:harnack}.

The case of general nonlocal operators will be explained in detail in Section~\ref{OJSDLN09ryihfjg9reSTRpsdlmvIJJ}.

In Section~\ref{sodjvlfORTYHSJDgh} we also recall that
the antisymmetric boundary
Harnack inequality is also related to a weak Harnack inequality for super-solutions and to a local boundedness for sub-solutions, which are of independent interest,
and we provide new proofs of these results by exploiting the methodology put forth in this paper.

\section{Reducing boundary estimates to interior estimates: a heuristic discussion}\label{SE:2}

The gist of the arguments that we exploit can be easily explained in the classical case of the Laplacian, in which all computations are straightforward, but already reveal a hidden higher dimensional geometric feature of the problem.

For this, we take coordinates~$y\in\R^k$ and~$z\in\R^{m}$. We use the notation~$\rho=|y|$ and, given a smooth function~$u$ in a given domain of~$\R^{m+1}$, for~$\lambda\in\R$ we set
$$ \tilde v(y,z):=|y|^\lambda u(|y|,z)=\rho^\lambda u(\rho,z).$$
Since, for every~$j\in\{1,\dots,k\}$,
$$ \rho\partial_{y_j} \rho=\frac12\partial_{y_j} \rho^2=\frac12\partial_{y_j}(y_1^2+\dots+y_k^2)=y_j,$$
and therefore
$$ \nabla_{y} \rho=\frac{y}\rho,$$
computing the Laplacian in cylindrical coordinates we find that
\begin{align*}
\Delta\tilde v&=\Delta_{y,z}\rho^\lambda \,u+\rho^\lambda \Delta_{y,z}u +2\nabla_{y,z} \rho^\lambda \cdot\nabla_{y,z} u\\&=
\Delta_y \rho^\lambda\, u+\rho^\lambda \big(\Delta_y u+\Delta_z u\big)+2\nabla_y \rho^\lambda\cdot\nabla_y u\\&=
(\lambda+k-2)\lambda\rho^{\lambda-2}u
+\rho^\lambda \big(\partial^2_\rho u+(k-1)\rho^{-1}\partial_\rho u+\Delta_zu\big)\\&\qquad+2\big(\lambda \rho^{\lambda-1}\nabla_y\rho\big)\cdot\big(\partial_\rho u\nabla_y\rho \big)
\\&=(\lambda+k-2)\lambda\rho^{\lambda-2}u+(2\lambda+k-1) \rho^{\lambda-1}\partial_\rho u+\rho^\lambda \partial^2_\rho u+\rho^\lambda \Delta_zu.
\end{align*}
The choice~$2\lambda+k-1=0$ allows one to cancel the first order term.
The additional choice~$\lambda+k-2=0$ makes the zero order term vanish. All in all, if~$k=3$ and~$\lambda=-1$, we find that
\begin{equation}\label{TBCW} \Delta\tilde v=\frac{\Delta u}\rho,\end{equation}
hence~$\tilde v$ is harmonic if and only if so is~$u$.

For us, the interest of~\eqref{TBCW} is that, for instance, it allows us to recover a boundary Harnack inequality for the Laplace
operator straightforwardly from the interior Harnack inequality.
Specifically, {\em if~$c$ is a bounded Borel function on $B_1^+:=B_1\cap\{x_1>0\}\subset\R^n$
and~$u$ is a nonnegative solution of~$-\Delta u + c u = 0$ in~$B_1^+$, then
\begin{equation}
\label{PRE:eq:odd:harnack}
 \sup_{x \in B_{1/2}^+} \frac{u(x)}{x_1} \le C \inf_{x \in B_{1/2}^+} \frac{u(x)}{x_1} ,
\end{equation}
where $C$ is a positive constant depending only on~$n$ and~$\|c\|_{L^\infty(B_1^+)}$}.

In our setting, the proof of~\eqref{PRE:eq:odd:harnack} can be obtained directly combining~\eqref{TBCW} and the classical Harnack inequality: namely, if we use the notation~$\tilde c(y,z):=c(|y|,z)$, we deduce from~\eqref{TBCW}  that
$$\big(-\Delta \tilde v + \tilde c \tilde v\big)|y| =-\Delta u + c u = 0$$
whenever~$(|y|,z)\in B_1$ and so, in particular, for all~$(y,z)\in\R^3\times\R^{n-1}$ such that~$|(y,z)|<1$.

This and the classical Harnack inequality lead to
$$ \sup_{|(y,z)|<1/2} \tilde v(y,z) \le C \inf_{|(y,z)|<1/2} \tilde v(y,z),$$
from which~\eqref{PRE:eq:odd:harnack} plainly follows.

The strategy that we follow in this note is precisely to adapt this method to more general settings of nonlocal type.
In this situation, additional symmetry structures have to be imposed on the solution, due to the nonlocal features of the operator.
Moreover, the analog of~\eqref{TBCW} requires a series of more subtle arguments, since cylindrical coordinates are typically not easy to handle in a nonlocal framework, due to remote point interactions: this difficulty will be overcome by exploiting 
some classical tools from harmonic analysis, as described in detail in the forthcoming Section~\ref{sec:bochner}.

\section{Bochner's relation and its consequences}
\label{sec:bochner}

\subsection{Notation}

Let~$\ell\in\N$.
As usual, a {\em solid harmonic polynomial} of degree $\ell$ is a homogeneous polynomial $V$ on $\R^n$ of degree $\ell$ (that is, $V(r x) = r^\ell V(x)$) satisfying the Laplace equation $\Delta V = 0$.

Below, we will also use symbols with a tilde to denote elements of $\R^{n + 2 \ell}$ and functions on $\R^{n + 2 \ell}$, while symbols without a tilde for vectors in~$\R^n$ and functions on~$\R^n$. That is, for typographical convenience, we will usually (i.e., unless differently specified) write
$x\in\R^n$ and $\tilde x\in\R^{n + 2 \ell}$, and similarly~$f:\R^n\to\R$ and~$\tilde{f}:\R^{n+2\ell}\to\R$.

We use $\fourier_n$ to denote the $n$-dimensional {\em Fourier transform}: if $f$ is an integrable function on $\R^n$, then
\[
 \fourier_n f(\xi) = \int_{\R^n} f(x) e^{2\pi i \xi\cdot x}\, dx .
\]

When we want to emphasize that we are taking the Fourier transform with respect to the variable~$x$
we also use the notation~$ \fourier_n^{(x)}$. For instance, if~$f:\R^{2n}=\R^n\times\R^n\to\R$, to be written as~$f(x,y)$, we have that, for each~$y\in\R^n$,
$$  \fourier_n^{(x)} f(\xi,y) = \int_{\R^n} f(x,y) e^{2\pi i \xi\cdot x} \,dx .$$
We are following here the
convention on the Fourier transform
on page 28 in~\cite{stein}.

We also consider $L_n$ to be a Fourier multiplier on $\R^n$ with radial symbol $\psi(|\xi|)$, that is
\begin{equation}\label{FOU:SI}
 \fourier_n L_n f(\xi)  = \psi(|\xi|) \fourier_n f(\xi)
\end{equation}
for every Schwartz function $f$
(see the forthcoming calculation in~\eqref{Coqkf:FR}
for an explicit computation of a Fourier symbol).

The case in which $\psi(r) = r^{2 s}$ corresponds to the fractional Laplacian $L_n = (-\Delta)^s$, but the contents of this section work for general $\psi$ with at most polynomial growth.

Let $\mathcal R$ be a rotation on $\R^n$ (that is, an orthogonal transformation of $\R^n$), and for a function $f$ on $\R^n$ denote $\mathcal R f(\xi) = f(\mathcal R \xi)$. By definition, we have the following transformation rule for the Fourier transform:
\[
 \fourier_n (\mathcal R f)(\xi) = \mathcal R (\fourier_n f)(\xi) .
\]
Radial functions are invariant under $\mathcal R$, and so
\[
 \psi(|\xi|) \fourier_n (\mathcal R f)(\xi) = \mathcal R \bigl[\psi(|\xi|) \fourier_n f(\xi)\bigr] .
\]
The inverse Fourier transform satisfies a similar transformation rule, so that
\[
 \fourier_n^{-1} \bigl[\psi(|\xi|) \fourier_n \mathcal R f(\xi)\bigr] = \mathcal R \bigl[ \fourier_n^{-1} \bigl[ \psi(|\xi|) \fourier_n f(\xi) \bigr]\bigr]
\]
whenever $f$ is, say, a Schwartz function. In view of~\eqref{FOU:SI}, this means that
\begin{equation}
\label{eq:commute}
 L_n \mathcal R f(\xi) = \mathcal R L_n f(\xi) ,
\end{equation}
that is, $L_n$ commutes with rotations.

Let us call a function $f$ on $\R^n$ isotropic with respect to the first $k$ coordinates, or briefly \emph{$k$-isotropic}, if $f(x) = f(y)$ whenever
\[
 |(x_1, x_2, \ldots, x_k)| = |(y_1, y_2, \ldots, y_k)| \quad \text{and} \quad (x_{k+1}, x_{k+2}, \ldots, x_n) = (y_{k+1}, y_{k+2}, \ldots, y_n) .
\]
We will only use this notion for $k = 3$: the proof of our main result involves $3$-isotropic functions in $\R^{n+2}$.

A function $f$ on $\R^n$ is said to be \emph{symmetric} (with respect to the first variable) if~$f(-x_1, x_2, \ldots, x_n) = f(x_1, x_2, \ldots, x_n)$ (that is, if~$f$ is even with respect to the first variable; this coincides with the notion of a $1$-isotropic function). Similarly, $f$ is said to be \emph{antisymmetric} (with respect to the first variable) if~$f(-x_1, x_2, \ldots, x_n) = -f(x_1, x_2, \ldots, x_n)$. Clearly, $f$ is symmetric if and only if $x_1 f(x)$ is antisymmetric.

If $f$ is a $k$-isotropic function on $\R^n$, then $f$ is invariant under every rotation $\mathcal R$ which only acts on the first $k$ coordinates and leaves the remaining $n - k$ coordinates unchanged: $\mathcal R f = f$. In view of~\eqref{eq:commute}, we have $\mathcal R(L_n f) = L_n(\mathcal R f) = L_n f$, that is, $L_n f$ is invariant under $\mathcal R$. In other words, $L_n$ maps $k$-isotropic Schwartz functions into $k$-isotropic functions.

\subsection{Fourier transforms of radial functions} Now
we recall the tools from harmonic analysis which will come in handy for the development of the theory.
The gist is the link of the Fourier transform of an $n$-dimensional radial function (multiplied by an appropriate
harmonic polynomial
of degree~$\ell$) with that
of its $(n+2\ell)$-dimensional counterpart.

For instance, roughly speaking, for any even function~$f$ on~$\R$, if~$\tilde f:\R^3\to\R$ has the same profile of~$f$ (i.e., $\tilde f(\tilde x)=f(|\tilde x|)$ for each~$\tilde x\in\R^3$), then
the Fourier transform of the one-dimensional function~$x f(x)$ coincides with~$i\xi$ times
the Fourier transform of the three-dimensional radial function~$\tilde f$ at~$\tilde\xi$, as long as~$\tilde\xi\in\R^3$
is such that~$|\xi|=|\tilde\xi|$.

This fact holds true in higher dimensions as well, with suitable modifications,
and goes under the name of Bochner's relation (or Hecke--Bochner identity). This will be our crucial tool and we now recall its precise statement:

\begin{theorem}[Bochner's relation; Corollary on page 72 in~\cite{stein}]
\label{thm:bochner}
Let $f$ and $\tilde{f}$ be two radial Schwartz functions on $\R^n$ and $\R^{n + 2 \ell}$, respectively, with the same profile function: $f(x) = \tilde{f}(\tilde{x})$ whenever $|x| = |\tilde x|$.

Let $V$ be a solid harmonic polynomial on $\R^n$ of degree $\ell$.

Then,
\[
 \fourier_n (V f)(\xi) = i^\ell V(\xi) \fourier_{n + 2 \ell} \tilde{f}(\tilde{\xi}) ,
\]
for every~$\xi\in\R^n$ and~$\tilde\xi\in\R^{n+2\ell}$ such that $|\tilde{\xi}| = |\xi|$.
\end{theorem}

In the statement of Theorem~\ref{thm:bochner}, the assumption on the profile function of~$f$ and~$\tilde f$ simply means that
there exists a suitable~$f_0:\R\to\R$ such that~$f(x)=f_0(|x|)$ and~$\tilde f(\tilde x)=f_0(|\tilde x|)$ for all~$x\in\R^n$ and~$\tilde x\in\R^{n+2\ell}$ (the function~$f_0$ is called in jargon ``profile function'').

In our setting, we will only need Theorem~\ref{thm:bochner} for $n= \ell = 1$ and $V(x) = x$. In this case Theorem~\ref{thm:bochner} states that if $f$ is an even function on $\R$ and $\tilde{f}(\tilde{x}) = f(|x|)$ is the corresponding radial function on $\R^3$, then
\begin{equation}
\label{eq:bochner:odd}
 \fourier_1[x f(x)](\xi) = i \xi \fourier_3 \tilde{f}(\xi, 0, 0) .
\end{equation}

We have the following immediate extension of Theorem~\ref{thm:bochner}.

\begin{lemma}
\label{lem:odd:fourier}
Consider a function~$f:\R^n=\R\times\R^{n-1}\to\R$,
to be written as~$f(x, y)$, where $x \in \R$ and $y \in \R^{n - 1}$.

Assume that~$f$ is a Schwartz function on $\R^n$ which is symmetric with respect to the variable $x$ (i.e., $f(-x, y) = f(x, y)$
for all~$x \in \R$ and $y \in \R^{n - 1}$).

Let~$\tilde f:\R^{n+2}=\R^3\times\R^{n-1}\to\R$,
to be written as~$\tilde{f}(\tilde{x}, y)$, where $\tilde{x} \in \R^3$ and $y \in \R^{n - 1}$, be a $3$-isotropic function given by
\[
 \tilde{f}(\tilde{x}, y) = f(|\tilde{x}|, y) .
\]
Then,
\[
 \fourier_n[x f(x, y)](\xi, \eta) = i \xi \fourier_{n + 2} \tilde{f}((\xi, 0, 0), \eta) . \qedhere
\]
\end{lemma}

\begin{proof}
Given $y \in \R^{n - 1}$, we apply Bochner's relation~\eqref{eq:bochner:odd} with respect to $x \in \R$ to find that
\[
 \fourier_1^{(x)}[x f(x, y)](\xi, y) = i \xi \fourier_3^{(\tilde{x})} \tilde{f}(\tilde{\xi}, y) .
\]
It is now sufficient to apply the Fourier transform with respect to $y \in \R^{n - 1}$ for each fixed $\xi \in \R$.
\end{proof}

From Lemma~\ref{lem:odd:fourier} we obtain the following result.

\begin{lemma}
\label{lem:odd:operator}
In the assumptions of Lemma~\ref{lem:odd:fourier}, for every~$x\in\R$ and~$y\in\R^{n-1}$ we have that
\[
 L_n[x f(x, y)] = x L_{n + 2} \tilde{f}((x, 0, 0), y) .
\]
\end{lemma}

\begin{proof}
Let us denote
\[
 \tilde{g}(\tilde{x}, y) = L_{n + 2} \tilde{f}(\tilde{x}, y)
\]
and
\[
 g(x, y) = \tilde{g}((x, 0, 0), y) = L_{n +2} \tilde{f}((x, 0, 0), y) .
\]
Since $L_{n + 2}$ maps $3$-invariant functions into $3$-invariant functions, $\tilde{g}$ is $3$-invariant.

Let now~$(\xi,\eta)\in\R\times\R^{n-1}$ and let~$\tilde\xi\in\R^3$ be such that~$|\xi| = |\tilde{\xi}|$. Using, in order, the definition of $L_n$
in~\eqref{FOU:SI}, Lemma~\ref{lem:odd:fourier}, once again the definition of $L_{n + 2}$ in~\eqref{FOU:SI}, the definition of~$\tilde{g}$, and once again Lemma~\ref{lem:odd:fourier}, we find that
\begin{align*}
 \fourier_n L_n[x f(x, y)](\xi, \eta) & = \psi(|(\xi, \eta)|)\;\fourier_n[x f(x, y)](\xi, \eta) \\
 & = \psi(|(\xi, \eta)|) \; i \xi \fourier_{n + 2} \tilde{f}(\tilde{\xi}, \eta) \\
 & = i \xi \; \psi(|(\tilde \xi, \eta)|) \; \fourier_{n + 2} \tilde{f}(\tilde{\xi}, \eta) \\
 & = i \xi \fourier_{n + 2} L_{n + 2} \tilde{f}(\tilde{\xi}, \eta) \\
 & = i \xi \fourier_{n + 2} \tilde{g}(\tilde{\xi}, \eta) \\
 & = \fourier_n[x g(x, y)](\xi, \eta) .
\end{align*}
Thus, by inverting the Fourier transform,
\[ L_n[x f(x, y)]=x g(x, y). \]
Combining this and the definition of $g$, we conclude that
\begin{equation*}
L_n[x f(x, y)]  
=x L_{n + 2} \tilde{f}((x, 0, 0), y),
\end{equation*}
as desired.
\end{proof}

In the same vein, one shows the following more general result (we omit the proof since it can be obtained via the same argument as above, by utilizing the general identity in Theorem~\ref{thm:bochner} instead of the ones specialized for the case~$V(x)=x$).

\begin{lemma}
\label{lem:solid:operator}
Let $V$ be a solid harmonic polynomial of degree $\ell$ on $\R^k$. Let $f(x, y)$, where $x \in \R^k$ and $y \in \R^{n - k}$, be a Schwartz function on $\R^n$ which is isotropic with respect to the first variable $x$: $f(x, y) = f(x', y)$ whenever $|x| = |x'|$. Let $\tilde{f}(\tilde{x}, y)$, where $\tilde{x} \in \R^{k + 2 \ell}$ and $y \in \R^{n - k}$, be the corresponding Schwartz function on $\R^{n + 2 \ell}$, given by
\[
 \tilde{f}(\tilde{x}, y) = f(x, y)
\]
whenever $|\tilde{x}| = |x|$.

Then
\[
 L_n[V(x) f(x, y)] = V(x) L_{n + 2 \ell} \tilde{f}((x, 0), y) ,
\]
where $(x, 0) \in \R^{k + 2 \ell}$ stands for the vector $x \in \R^k$ padded with $2 \ell$ zeroes.
\end{lemma}

%
%

\section{Fractional Laplacians}
\label{sec:flap}

\subsection{Notation}\label{sec:notation}
Now we
specialize the result of Section~\ref{sec:bochner}
to the case of the fractional Laplacian of order~$s\in(0,1)$, corresponding to the choice~$\psi(r) = r^{2 s}$ in~\eqref{FOU:SI}. To this end, we introduce some further notation.

If~$f$ and~$g$ are functions on~$\R^n$, we write
$$ \langle f,g\rangle_{\R^n}=\int_{\R^n}f(x)g(x)\,dx.$$

We use the notation~$f\in C^\alpha$ for non-integer~$\alpha>0$, meaning
that~$f\in C^{k,\beta}$ where~$\alpha=k+\beta$ with~$k\in\N$ and~$\beta\in(0,1)$.

\subsection{Fractional Laplacian identities and a Harnack inequality}\label{subsec:fl}
We now translate the results of Section~\ref{sec:bochner} into the notation that we have just introduced.

\begin{corollary}
\label{cor:odd:flap}
Let $f$ be a 
Schwartz function on $\R^n$, symmetric with respect to the first variable. Let $\tilde{f}$ be the corresponding 3-isotropic Schwartz function on $\R^{n + 2}$, given by
\[
 \tilde{f}(\tilde{x}) = f\bigg (\sqrt{\tilde{x}_1^2 + \tilde{x}_2^2 + \tilde{x}_3^2}, \tilde{x}_4, \ldots, \tilde{x}_{n + 2} \bigg ) .
\]
Finally, let $g(x) = x_1 f(x)$. Then
\[
 (-\Delta)^s g(x) = x_1 (-\Delta)^s \tilde{f}((x_1, 0, 0), x_2, \ldots, x_n) .
\]
\end{corollary}

\begin{proof} This follows directly from Lemma~\ref{lem:odd:operator}: indeed,
choosing~$\psi(r) = r^{2 s}$ in~\eqref{FOU:SI} gives that 
both~$ L_n$ and~$ L_{n+2}$ are the fractional Laplacians in their respective dimensions
and accordingly, for every~$x=(x_1,x_2,\dots,x_n)\in\R^n$,
\begin{equation*}
(-\Delta)^sg(x)=(-\Delta)^s[x_1 f(x)]=x_1(-\Delta)^s
\tilde f((x_1,0,0),x_2,\dots,x_n)
.\qedhere
\end{equation*}
\end{proof}

We point out that an alternative approach
to the proof of Corollary~\ref{cor:odd:flap} is to rely on
the Caffarelli-Silvestre extension problem
by checking that the extension of the function~$g$ at~$ (x_1, x_2, \dots , x_n, z)$ coincides with~$x_1$ multiplied by the extension of the function~$f$ at~$ (x_1, 0, 0, x_2, \dots , x_n, z)$, where~$z$ denotes the extension variable.

\begin{remark}
The identity in Corollary~\ref{cor:odd:flap} can be alternatively proved by a  direct calculation involving the formula
\begin{align}
 & r \int_{\sph^2} (\alpha^2 + (r z_1 - \beta)^2 + (r z_2)^2 + (r z_3)^2)^{-1 - \gamma} \, \mathcal H^2 (dz) \nonumber \\
 & \qquad = \frac{\pi}{ \beta \gamma} \, \bigl( (\alpha^2 + (r - \beta)^2)^{-\gamma} - (\alpha^2 + (r + \beta)^2)^{-\gamma} \bigr) , \label{WpnIQVKB}
\end{align}
where $\sph^2$ is the unit sphere in $\R^3$ and $\mathcal H^2$ is the surface measure.
One substitutes $\alpha = \sqrt{(x_2 - y_2)^2 + \ldots + (x_n - y_n)^2}$, $\beta = x_1$, $r = |y_1|$, and $\gamma = \tfrac{n}{2} + s$ and then integrates with respect to $y \in \R^n$.

One can prove~\eqref{WpnIQVKB}, for example, via the smooth co-area formula. Indeed, \begin{align*}
   &\hspace{-2em}\int_{\sph^2} (\alpha^2 + (r z_1 - \beta)^2 + (r z_2)^2 + (r z_3)^2)^{-1 - \gamma} \, \mathcal H^2 (dz)\\
   &= \int_{\sph^2} (\alpha^2 + \beta^2+r^2 -2r\beta z_1)^{-1 - \gamma} \, \mathcal H^2 (dz) \\
    &= \int_{-1}^1\int_{\sph^2\cap \{z_1=t\}}(\alpha^2 + \beta^2+r^2 -2r\beta z_1)^{-1 - \gamma}(1-z_1^2)^{-\frac12} \, \mathcal H^1 (dz) \, dt \\
    &= 2\pi \int_{-1}^1(\alpha^2 + \beta^2+r^2 -2r\beta t)^{-1 - \gamma} \, dt \\
    &= \frac\pi{r\beta\gamma} \big [ (\alpha^2 + (r-\beta)^2)^{- \gamma}-(\alpha^2 +(r+\beta)^2 )^{- \gamma} \big ],
\end{align*} as required (alternatively,
instead of
using the co-area formula, one can parametrize $z_1 = t$, $z_2 = \sqrt{1 - t^2} \cos s$, $z_3 = \sqrt{1 - t^2} \sin s$).
\end{remark}

An extension of Corollary~\ref{cor:odd:flap} to arbitrary antisymmetric functions is rather straightforward: in particular, in item~\ref{thm:odd:flap:c} below we consider the extension of the fractional Laplace operator on the class of antisymmetric functions, as described in~\cite{dtv}.

\begin{theorem}
\label{thm:odd:flap}
Suppose that $\Omega$ is an open set of~$\R^n$ and $\delta>0$. 
Let $u$ be a function on $\R^n$ such that $u$ is~$C^{2s + \delta}$ in~$\Omega$ and
antisymmetric, with
\begin{equation}
\int_{\R^n}\frac{|x_1|\,|  u(x)|}{(1 + |x|)^{n + 2 +2s}} \,dx<+\infty. \label{K4ptw5NL}
\end{equation}
Let $\tilde{v}$ be the corresponding function on $\R^{n + 2}$ given by\footnote{Strictly speaking, the function~$\tilde v$ is not defined when~$(\tilde x_1, \tilde x_2, \tilde x_3) = (0, 0, 0)$, however the case of interest in our setting occurs when~$0\in\Omega$ and then the regularity
and antisymmetry of~$u$ yields that
\begin{eqnarray*}&& \lim_{(\tilde x_1, \tilde x_2, \tilde x_3)\to(0,0,0)}\tilde{v}(\tilde{x}) = \lim_{(\tilde x_1, \tilde x_2, \tilde x_3)\to(0,0,0)}\frac{u\big (\sqrt{\tilde{x}_1^2 + \tilde{x}_2^2 + \tilde{x}_3^2}, \tilde{x}_4, \ldots, \tilde{x}_{n + 2} \big )}{\sqrt{\tilde{x}_1^2 + \tilde{x}_2^2 + \tilde{x}_3^2}}\\&&\quad\qquad\qquad=
\lim_{t\to0^+}\frac{u\big (t, \tilde{x}_4, \ldots, \tilde{x}_{n + 2} \big )}{t}=\partial_1
u\big(0,\tilde{x}_4, \ldots, \tilde{x}_{n + 2} \big ).\end{eqnarray*}
We will indulge in this small abuse of notation in the rest of this paper as well.}
\begin{equation}\label{q0wifpojgl5tDD23rt-1}
 \tilde{v}(\tilde{x}) = \frac{u\big (\sqrt{\tilde{x}_1^2 + \tilde{x}_2^2 + \tilde{x}_3^2}, \tilde{x}_4, \ldots, \tilde{x}_{n + 2} \big )}{\sqrt{\tilde{x}_1^2 + \tilde{x}_2^2 + \tilde{x}_3^2}} \, .
\end{equation}
Then:
\begin{enumerate}[label={\rm (\alph*)}]
\item\label{thm:odd:flap:a} if $\Omega_0 = \{x \in \Omega : x_1 \ne 0\}$, then the function $\tilde{v}$ is $C^{2s + \delta}_{\mathrm{loc}}$ in
\[
 \tilde \Omega_0 = \bigg \{\tilde{x} \in \R^{n + 2} : \Big (\sqrt{\tilde{x}_1^2 + \tilde{x}_2^2 + \tilde{x}_3^2}, \tilde{x}_4, \ldots, \tilde{x}_{n + 2}\Big ) \in \Omega_0\bigg \} ;
\]
\item\label{thm:odd:flap:b} we have that
\begin{equation*}
\int_{\R^{n+2}}\frac{|\tilde{v}(\tilde{x})|}{
(1 + |\tilde{x}|)^{n + 2 + 2s}}\,d\tilde x<+\infty;\end{equation*}
\item\label{thm:odd:flap:c} for every $x \in \Omega_0$, we have that\footnote{In general, without the antisymmetry assumption,~\eqref{K4ptw5NL} is not a strong enough assumption on the decay of \(u\) for the fractional Laplacian to be well-defined. To rectify this issue, we have to use the antisymmetric fractional Laplacian, defined according to \cite[Definition 1.2]{dtv}.}
\begin{equation}
\label{eq:odd:flap}
 (-\Delta)^s u(x) = x_1 (-\Delta)^s \tilde{v}((x_1, 0, 0), x_2, \ldots, x_n) ;
\end{equation}
\item\label{thm:odd:flap:d} for an arbitrary $C_c^\infty$ function $g$ on $\Omega$, we have
\begin{equation}
\label{eq:odd:flap:weak}
 \langle u, (-\Delta)^s g \rangle_{\R^n} = (2 \pi)^{-1} \langle \tilde{v}, (-\Delta)^s \tilde{f} \rangle_{\R^{n + 2}} ,
\end{equation}
where
\[
 \tilde f(\tilde x)= \frac{ g_A\big (\sqrt{\tilde{x}_1^2 + \tilde{x}_2^2 + \tilde{x}_3^2}, \tilde{x}_4, \ldots, \tilde{x}_{n + 2} \big )}{\sqrt{\tilde{x}_1^2 + \tilde{x}_2^2 + \tilde{x}_3^2}}
\]
and
\[
g_A(x) = \frac 12 \big (  g(x_1,\dots,x_n) - g(-x_1,x_2,\dots,x_n) \big ) 
\]
is the anti-symmetric part of $g$;
\item\label{thm:odd:flap:e} for an arbitrary $C_c^\infty$ function $\tilde{f}$ on 
\[\tilde \Omega=\bigg \{\tilde{x} \in \R^{n + 2} : \Big (\sqrt{\tilde{x}_1^2 + \tilde{x}_2^2 + \tilde{x}_3^2}, \tilde{x}_4, \ldots, \tilde{x}_{n + 2}\Big ) \in \Omega\bigg \} ,
\]
we have
\begin{equation}
\label{eq:odd:flap:reverse}
\langle \tilde{v}, (-\Delta)^s \tilde{f} \rangle_{\R^{n + 2}} = 2 \pi \langle u, (-\Delta)^s g \rangle_{\R^n} ,
\end{equation}
where $g(x) = x_1 f(x)$ and $f$ is given in terms of the 3-isotropic symmetrization of $\tilde{f}$ by
\begin{equation}\label{ISODEF}
 f(x) = \frac{1}{4 \pi} \int_{\sph^2} \tilde{f}(|x_1| z, x_2, \ldots, x_n)\, \mathcal H^2(dz) ;
\end{equation}
here $\sph^2$ is the unit sphere in $\R^3$ and $\mathcal H^2$ is the surface measure.
\end{enumerate}
\end{theorem}

\begin{remark}
Note that, in~Theorem~\ref{thm:odd:flap}~\ref{thm:odd:flap:a}, the conclusion $\tilde v\in C^{2s+\delta}_{\mathrm{loc}}(\tilde \Omega_0)$ cannot be relaxed to $\tilde v\in C^{2s+\delta}(\tilde \Omega_0)$. Indeed, consider $\Omega = (0,1) \times \R^{n-1}$ and $u(x) = x_1^{2s+\delta}$ for $x_1\geq 0$ extended antisymmetrically to all of $\R^n$. Then $\Omega_0=\Omega$, $\tilde \Omega_0 =\big (  B_1^3\setminus \{0\}\big ) \times \R^{n-1}$ (here $B^3_1$ denotes the unit ball in $\R^3$), and $u\in C^{2s+\delta}(\Omega_0)$, but \begin{align*}
\tilde v(\tilde x ) = \big ( \tilde x_1^2+\tilde x_2^2+\tilde x_3^2 \big ) ^{\frac{2s+\delta-1} 2 } 
\end{align*} which is not in $C^{2s+\delta} (\tilde \Omega_0)$. 
\end{remark}

\begin{proof}[Proof of Theorem~\ref{thm:odd:flap}]
To prove assertion~\ref{thm:odd:flap:a}, observe that, for all $\tilde x \in \tilde \Omega_0$, \begin{align*}
\sqrt{\tilde x_1^2+\tilde x_2^2+\tilde x_3^2} 
\neq 0 .
\end{align*} Hence, the map from $\tilde \Omega_0 \to \R $ defined by $\tilde x \mapsto \frac 1 {\sqrt{\tilde x_1^2+\tilde x_2^2+\tilde x_3^2}}$, is in $C^\infty_{\mathrm{loc} } (\tilde \Omega_0)$. Moreover, 
\[
\tilde x \mapsto u\Big (\sqrt{\tilde{x}_1^2 + \tilde{x}_2^2 + \tilde{x}_3^2}, \tilde{x}_4, \ldots, \tilde{x}_{n + 2}\Big )
\]
is in $C^{2s+\delta}(\tilde \Omega_0)$ which implies $\tilde v\in C^{2s+\delta}_{\mathrm{loc}}(\tilde \Omega_0)$. This proves~\ref{thm:odd:flap:a}.

Additionally,
\begin{equation}\label{4YintRdt1}\begin{split}&{\mathcal{I}}:=
\int_{\R^{n+2}} \frac{|\tilde{v}(\tilde{x})|}{(1 + |\tilde{x}|)^{n + 2 + 2s} }\, d\tilde x=
\int_{\R^{n+2}}
\frac{ \Bigl|u\Bigl(\sqrt{\tilde{x}_1^2 + \tilde{x}_2^2 + \tilde{x}_3^2}, \tilde{x}_4, \ldots, \tilde{x}_{n + 2}\Bigr)\Bigr|}{\sqrt{\tilde{x}_1^2 + \tilde{x}_2^2 + \tilde{x}_3^2} \;\big(1 + |\tilde{x}|\big)^{n + 2 + 2s}} \,d\tilde x\\& \qquad
= \int_{\R^3} \int_{\R^{n-1}} \frac{ \bigl|u\bigl( |a|,b\bigr)\bigr|}{|a|\;\big(1 + \sqrt{|a|^2+|b|^2}\big)^{n + 2 + 2s}} \,da\, db
,\end{split}\end{equation}
Now we observe that if~$\Phi:\R\to\R$ then, integrating in spherical coordinates, we get
\begin{equation}\label{UT:01}
\int_{\R^3} \Phi(|x|)\,dx = 4\pi\int_0^{+\infty}\tau^2 \Phi(\tau) \,d\tau.
\end{equation}
Now, gathering~\eqref{4YintRdt1} and~\eqref{UT:01}, we see that
\[
{\mathcal{I}}
= {4\pi}\int_0^{+\infty} \int_{\R^{n-1}}\frac{\tau \, |u(\tau,b)|}{\big(1 + \sqrt{|\tau|^2+|b|^2}\big)^{n + 2 + 2s}}\, db\, d\tau 
= {4\pi}\int_{ \R^{n}} \frac{|x_1|\, |u(x)|}{ (1 + |x|)^{ n + 2 + 2s}}\, dx ,
\]
from which~\ref{thm:odd:flap:b} plainly follows.

To prove the claim in~\ref{thm:odd:flap:c}, we notice that if $f(x) = x_1^{-1} u(x)$ is a Schwartz function, item~\ref{thm:odd:flap:c} reduces to Corollary~\ref{cor:odd:flap}. The extension to general $u$ follows by mollification and cutting off. This approximation procedure is not completely standard in our scenario, so we have provided the details.

Let
\begin{align*}
    \mathscr A_s(\R^n) & = \{ u : \R^n \to \R \text{ s.t. } u \text{ is measurable, antisymmetric, and } \| u\|_{  \mathscr A_s(\R^n) }<+\infty \} , \\
    \mathscr L_s(\R^{n + 2}) & = \{ \tilde{v} : \R^{n + 2} \to \R \text{ s.t. } \tilde{v} \text{ is measurable and } \| \tilde{v}\|_{  \mathscr L_s(\R^{n+2}) }<+\infty \} ,
\end{align*}
where
\begin{equation}\label{AAAnorm}
\begin{split}
    \| u\|_{  \mathscr A_s(\R^n) } & ={4\pi} \int_{\R^n}\frac{|x_1|\,|  u(x)|}{(1 + |x|)^{n + 2 +2s}} \,dx , \\
    \| \tilde{v}\|_{\mathscr L_s(\R^{n+2})} & = \int_{\R^{n+2}}\frac{\vert\tilde{v}(\tilde{x})\vert}{(1 + |\tilde{x}|)^{n + 2 +2s}} \,d\tilde{x} .
\end{split}\end{equation}
Let $\eta_\varepsilon$ be a standard mollifier on $\R^{n + 2}$, that is $\eta_1$ is a smooth, nonnegative, compactly supported function with integral $1$, and $\eta_\varepsilon(\tilde x) = \varepsilon^{-n - 2} \eta_1(\tilde x / \varepsilon)$. We have $u \in \mathscr A_s(\R^n)$, and $\tilde v \in \mathscr L_s(\R^{n + 2})$ by~\ref{thm:odd:flap:b}. Define
\begin{align*}
 \tilde{v}_k(\tilde x) & = (\tilde{v} \ast \eta_{1/k})(\tilde x) \eta_k(\tilde x) ,
\end{align*}
and
\begin{align*}
 u_k(x) & = x_1 \tilde{v}_k((x_1, 0, 0), x_2, x_3, \ldots, x_n) .
\end{align*}
Note that $\tilde v_k$ and $u_k$ are smooth and have compact support, so, in particular, they are Schwartz functions. By a standard estimate\footnote{The proof is very similar to the one for standard $L^1$ spaces. When $\tilde v$ is continuous and compactly supported, an explicit calculation gives a bound for $\|\tilde v_k - \tilde v\|_{\mathscr L_s(\R^{n + 2})}$ in terms of the modulus of continuity of $\tilde v$. For a general $\tilde v$, one approximates $\tilde v$ by a continuous, compactly supported function in $\mathscr L_s(\R^{n + 2})$. For the sake of completeness, we have included the proof in Appendix~\ref{ZVk7DrGw}.}, we have $\tilde v_k \in \mathscr L_s(\R^{n + 2})$ and
\begin{align*}
 \|\tilde v_k - \tilde v\|_{\mathscr L_s(\R^{n + 2})} \to 0
\end{align*}
as $k \to \infty$. We have seen in the proof of~\ref{thm:odd:flap:b} that for an arbitrary antisymmetric function $u$ on $\R^n$ and the corresponding function $\tilde v$ on $\R^{n + 2}$, $u \in \mathscr A_s(\R^n)$ if and only if $\tilde v \in \mathscr L_s(\R^{n + 2})$, and in this case $\|u\|_{\mathscr A_s(\R^n)} = \|\tilde v\|_{\mathscr L_s(\R^{n + 2})}$. In particular, since $\tilde v_k \in \mathscr L_s(\R^{n + 2})$, we have $u_k \in \mathscr A_s(\R^n)$, and additionally
\begin{align*}
 \|u_k - u\|_{\mathscr A_s(\R^n)} & = \|\tilde v_k - \tilde v\|_{\mathscr L_s(\R^{n + 2})} .
\end{align*}
Furthermore, since \(x_1^{-1}u \in  C^{2s+\delta}_{\mathrm{loc}}(\Omega_0)\), we have that \(u_k \to u\) in \( C^{2s+\delta}_{\mathrm{loc}}(\Omega_0)\). These two properties and \cite[Lemma 2.3]{dtv} imply that \((-\Delta)^su_k \to (-\Delta)^su\) in \(\Omega_0\). Next, \(\tilde v_k \to v\) in \(C^{2s+\delta}_{\mathrm{loc}}(\tilde \Omega_0)\) using~\ref{thm:odd:flap:a}. Moreover,
\begin{align*}
 \| \tilde v_k - \tilde v\|_{\mathscr L_s(\R^{n+2})} \to 0
\end{align*}
as $k \to \infty$, so standard estimates for the fractional Laplacian imply \((-\Delta)^s\tilde v_k \to (-\Delta)^s\tilde v\) in \(\tilde \Omega_0\). Using~\ref{thm:odd:flap:c} for Schwartz functions $u_k$ and $\tilde v_k$, and combining it with the above two convergence results, we conclude that~\ref{thm:odd:flap:c} holds in the general case.

We now prove assertions~\ref{thm:odd:flap:d} and~\ref{thm:odd:flap:e}. For~\ref{thm:odd:flap:d}, since $\tilde f$ is $3$-isotropic, also $(-\Delta)^s \tilde f$ is $3$-isotropic:
\begin{align*}
(-\Delta)^s \tilde f (\tilde x) &= (-\Delta)^s \tilde f \bigg ( \Big (\sqrt{\tilde x_1^2 + \tilde x_2^2 + \tilde x_3^2} , 0 , 0 \Big ),\tilde x_4, \dots ,\tilde x_{n+2} \bigg )
, 
\end{align*} so Corollary~\ref{cor:odd:flap} implies that \begin{align*}
(-\Delta)^s \tilde f (\tilde x) &= \frac{ (-\Delta)^s g_A \big (\sqrt{\tilde x_1^2 + \tilde x_2^2 + \tilde x_3^2}, \tilde x_4, \ldots, \tilde x_{n+2} \big )} {\sqrt{\tilde x_1^2 + \tilde x_2^2 + \tilde x_3^2}} .
\end{align*} Hence, using the definition of \(\tilde v\), we have  \begin{align*}
& \hspace{-1em} \langle \tilde v , (-\Delta)^s \tilde f \rangle_{\R^{n+2}} \\&= \int_{\R^{n+2}} \frac{u \big (\sqrt{\tilde x_1^2 + \tilde x_2^2 + \tilde x_3^2}, \tilde x_4, \ldots, \tilde x_{n+2} \big ) (-\Delta)^s g_A \big (\sqrt{\tilde x_1^2 + \tilde x_2^2 + \tilde x_3^2}, \tilde x_4, \ldots, \tilde x_{n+2} \big )}{\tilde x _1 ^2+\tilde x _2 ^2+\tilde x _3 ^2 }\, d \tilde x \\
& = \int_{\R^3} \int_{\R^{n-1}} \frac{u(|a|, b) (-\Delta)^s g_A(|a|, b)}{|a|^2}\, db \, da.
\end{align*}
Integration in spherical coordinates leads to
\begin{align*}
\langle \tilde v , (-\Delta)^s \tilde f \rangle_{\R^{n+2}}
&= 4 \pi \int_0^{+\infty} \int_{ \R^{n-1}}  u (\tau,b) (-\Delta)^s g_A (\tau,b) \, d b\,d\tau. 
\end{align*} Then, since both \(u\) and \((-\Delta)^s g_A\) are antisymmetric, \begin{align*}
\langle \tilde v , (-\Delta)^s \tilde f \rangle_{\R^{n+2}} &= 2 \pi \int_{-\infty}^{+\infty}\int_{ \R^{n-1} }  u (\tau, b) (-\Delta)^s g_A (\tau, b) \, d b\,d\tau \\
&= 2 \pi \langle u , (-\Delta)^s g \rangle_{\R^n}. 
\end{align*}

For~\ref{thm:odd:flap:e}, we begin by assuming that~\(\tilde f\) is \(3\)-isotropic and deal with the general case later. Since \(\tilde f\) is \(3\)-isotropic, Corollary~\ref{cor:odd:flap} applies to $\tilde f$, $f$ and $g$: we have
\[
 (-\Delta)^s \tilde{f}(\tilde x) = \frac{(-\Delta)^s g\big (\sqrt{\tilde x_1^2 + \tilde x_2^2 + \tilde x_3^2}, \tilde x_4, \ldots, \tilde x_{n+2} \big )}{\sqrt{\tilde x_1^2 + \tilde x_2^2 + \tilde x_3^2}} \, .
\]
Using spherical coordinates,
\begin{align*}
    &\hspace{-2em}\langle \tilde v , (-\Delta)^s \tilde f \rangle_{\R^{n+2}} \\
    &= \int_{\R^{n+2}} \frac{u \big ( \sqrt{\tilde x_1^2+\tilde x_2^2+\tilde x_3^2},\tilde x_4, \ldots, \tilde x_{n+2}\big ) \, (-\Delta)^s g\big (\sqrt{\tilde x_1^2 + \tilde x_2^2 + \tilde x_3^2}, \tilde x_4, \ldots, \tilde x_{n+2} \big ) }{\tilde x_1^2+\tilde x_2^2+\tilde x_3^2} \, d \tilde x \\
    & = \int_{\R^3} \int_{\R^{n-1}} \frac{u(|a|, b) (-\Delta)^s g(|a|, b)}{|a|^2} \, d b\, d a \\
    &= 4\pi \int_0^{+\infty} \int_{\R^{n-1}} u(\tau, b) (-\Delta)^s g(\tau, b)\,  d b\, d\tau.
\end{align*}
Since both $u$ and $(-\Delta)^s g$ are symmetric, we have
\[
 \langle \tilde v, (-\Delta)^s \tilde f \rangle_{\R^{n+2}} = 2 \pi \int_{-\infty}^{+\infty} \int_{\R^{n-1}} u(\tau, b) (-\Delta)^s g(\tau, b) \, db \, d\tau = 2 \pi \langle u, (-\Delta)^s g \rangle_{\R^n} .
\]

Now, if \(\tilde f\) is not \(3\)-isotropic, let
\[
 \tilde F(\tilde x) = \frac{1}{4 \pi} \int_{\sph^2} \tilde{f}\Big (z \sqrt{\tilde x_1^2 + \tilde x_2^2 + \tilde x_3^2}, \tilde x_4, \ldots, \tilde x_{n+2} \Big ) \, \mathcal H^2(dz)
\]
be the $3$-isotropic projection of $\tilde f$. By the definition of $f$,
\[
 \tilde F(\tilde x) = f\Big (\sqrt{\tilde x_1^2 + \tilde x_2^2 + \tilde x_3^2}, \tilde x_4, \ldots, \tilde x_{n+2} \Big ) ,
\]
and so, since we already proved the result for $3$-isotropic functions, we have
\[
 \langle \tilde v, (-\Delta)^s \tilde F \rangle_{\R^{n+2}} = 2 \pi \langle u, (-\Delta)^s g \rangle_{\R^n} .
\]
Therefore, it remains to prove that
\begin{equation}
\label{eq:orthogonal}
 \langle \tilde v, (-\Delta)^s (\tilde f - \tilde F) \rangle_{\R^{n+2}} = 0 .
\end{equation}
Denote $\tilde G = \tilde f - \tilde F$. Then, by the definition of $\tilde F$, 
\[
 \int_{\sph^2} \tilde G(\tau z, \tilde x_4, \ldots, \tilde x_{n+2}) \, \mathcal H^2(dz) = 0
\]
for every $\tau > 0$ and $(\tilde x_4, \ldots, \tilde x_{n+2}) \in \R^{n-1}$. Thus, $\tilde G$ is orthogonal to every $3$-isotropic function: if $\tilde u$ is $3$-isotropic, then integration in spherical coordinates gives
\begin{align*}
 \langle \tilde u, \tilde G\rangle & = \int_{\R^3} \int_{\R^{n-1}} \tilde u(a, b) \tilde G(a, b) \, db \, da \\
 & = \int_{\R^3} \int_{\R^{n-1}} \tilde u((|a|, 0, 0), b) \tilde G(a, b) \, db \, da \\
 & = \int_0^{+\infty} \int_{\sph^2} \int_{\R^{n-1}} \tau^2 \tilde u((\tau, 0, 0), b) \tilde G(\tau z, b) \, db \, \mathcal H^2(dz) \, d\tau \\
 & = \int_0^{+\infty} \int_{\R^{n-1}} \tau^2 \tilde u((\tau, 0, 0), b) \biggr(\int_{\sph^2} \tilde G(\tau z, b) \, \mathcal H^2(dz)\biggr) db \, d\tau = 0 .
\end{align*}
We already know that $(-\Delta)^s$ maps $3$-isotropic Schwartz functions into $3$-isotropic functions. By approximation, the class of $3$-isotropic functions in $L^2(\R^{n+2})$ forms an invariant subspace of $(-\Delta)^s$. Therefore, also its orthogonal complement in $L^2(\R^n)$ is invariant. Since $\tilde G$ is orthogonal to the class of $3$-isotropic functions, we conclude that $(-\Delta)^s \tilde G$ has the same property, and~\eqref{eq:orthogonal} follows.
\end{proof}

In our framework, a natural
application\footnote{When~$s=1$, Theorem~\ref{thm:odd:harnack}
boils down to~\eqref{PRE:eq:odd:harnack}.} of Theorem~\ref{thm:odd:flap} 
leads to the proof of Theorem~\ref{thm:odd:harnack}:

\begin{proof}[Proof of Theorem~\ref{thm:odd:harnack}]
We use
the notation of Theorem~\ref{thm:odd:flap}
and we point out that~$\tilde{v}$ is everywhere nonnegative, due to the assumptions on~$u$ and the definition of~$\tilde v$
in~\eqref{q0wifpojgl5tDD23rt-1}.

By item~\ref{thm:odd:flap:c} of Theorem~\ref{thm:odd:flap}, we have that
\[
 (-\Delta)^s u(x) = x_1 (-\Delta)^s \tilde{v}((x_1, 0, 0), x_2, \ldots, x_n)
\]
for $x \in \Omega_0$, or equivalently\footnote{One can compare~\eqref{TBCW}
with~\eqref{TBCW2}. Namely, \eqref{TBCW2} recovers~\eqref{TBCW} when~$s=1$ with the notation~$y=(\tilde x_1,\tilde x_2,\tilde x_3)$ and
$z=(\tilde x_4,\dots,\tilde x_{n+2})$.}
\begin{equation}
\label{TBCW2}
 \sqrt{\tilde{x}_1^2 + \tilde{x}_2^2 + \tilde{x}_3^2} \, (-\Delta)^s \tilde{v}(\tilde{x}) = (-\Delta)^s u \Big (\sqrt{\tilde x_1^2 + \tilde x_2^2 + \tilde x_3^2}, \tilde x_4, \ldots, \tilde x_{n+2} \Big ) ,
\end{equation}
for $\tilde{x} \in \tilde \Omega_0$ (recall that $(-\Delta)^3 \tilde v$ is $3$-isotropic).

Therefore, if $x = \big  (\sqrt{\tilde x_1^2 + \tilde x_2^2 + \tilde x_3^2}, \tilde x_4, \ldots, \tilde x_{n+2} \big )$, we have
\[
 \sqrt{\tilde{x}_1^2 + \tilde{x}_2^2 + \tilde{x}_3^2} \Bigl( (-\Delta)^s \tilde{v}(\tilde{x}) + c(x) \tilde{v}(\tilde{x}) \Bigr) = (-\Delta)^s u(x) + c(x) u(x) = 0 ,
\]
that is, setting $\tilde{c}(\tilde{x}) = c(x)$,
\begin{equation}
\label{eq:schroedinger}
 (-\Delta)^s \tilde{v}(\tilde{x}) + \tilde{c}(\tilde{x}) \tilde{v}(\tilde{x}) = 0
\end{equation}
in $\tilde \Omega_0$. In fact, the above identity holds in the weak sense in $\tilde \Omega$: this follows directly from item~\ref{thm:odd:flap:e} of Theorem~\ref{thm:odd:flap}.

Since $\tilde{v}$ is nonnegative in $\R^{n + 2}$, the standard Harnack inequality for weak solutions of the Schrödinger equation~\eqref{eq:schroedinger} (see e.g.~\cite[Theorem~1.1]{MR2825646})
implies that for every compact subset $\tilde{K}$ of $\tilde \Omega$, we have
\[
 \esssup_{\tilde{K}} \tilde{v} \le C(\tilde K, \tilde \Omega, \|\tilde c\|_{L^\infty(\tilde \Omega})) \essinf_{\tilde{K}} \tilde{v} .
\]
Applying this to $\tilde{K} = \big \{\tilde x \in \R^{n+2} : \big (\sqrt{\tilde x_1^2 + \tilde x_2^2 + \tilde x_3^2}, \tilde x_4, \ldots, \tilde x_{n+2} \big ) \in K\big \}$ with $K$ a compact subset of $\Omega$, we obtain
\[
 \esssup_{x \in K} \frac{u(x)}{x_1} \le C(K, \Omega, \|c\|_{L^\infty(\Omega_+)}) \essinf_{x \in K} \frac{u(x)}{x_1} ,
\]
which is equivalent to~\eqref{eq:odd:harnack}.
\end{proof}

\section{More general nonlocal operators}\label{OJSDLN09ryihfjg9reSTRpsdlmvIJJ}
In this section, we generalize the results obtained in the specific setting of the fractional Laplacian to the more general class of operators $L_n$ presented in~\eqref{5XOlUbY7}.

Given an open set \(\Omega \subset \R^n\) and a function in \(u\in C^2_{{\rm{loc}}}(\Omega) \cap L^\infty(\R^n)\), by a standard computation (see, for example, the computation in \cite[p.~9]{BUC}
), \(L_n u(x)\) is defined for all \(x\in \Omega\) and, moreover, $L_nu$ is bounded on compact subsets of $\Omega$
. Furthermore, the Fourier symbol of \(L_n\) is
\begin{equation}\begin{split}\label{Coqkf:FR}
 \Psi_n (\xi) & {} = \int_{\R^n} \big ( 1 - e^{i \xi \cdot x} + \xi \cdot x \chi_{(0,1)}(|x|) \big ) \, \jj_n(|x|) \, dx \\
 &= \int_0^\infty \int_{\sph^{n-1}} \big ( 1
 - e^{i r \xi \cdot \omega } + i r(\xi \cdot \omega ) \chi_{(0,1)}(r) \big ) \jj_n(r) r^{n - 1}  \,  \mathcal H^{n-1}(d\omega) \, dr \\
    &= \frac12 \int_0^\infty \int_{\sph^{n-1}} \big ( 2 - e^{i r \xi \cdot \omega } -e^{-i r \xi \cdot \omega }  \big ) \jj_n(r) r^{n - 1} \, \mathcal H^{n-1}(d\omega) \, dr \\
    &=\int_0^\infty \int_{\sph^{n-1}} \big ( 1 - \cos( r \xi \cdot \omega  ) \big ) \jj_n(r) r^{n - 1} \,  \mathcal H^{n-1}(d\omega) \, dr \\ 
    &= \mathcal H^{n-2}(\sph^{n-2}) \int_0^\infty \int_{-1}^1 \big ( 1 - \cos( r s |\xi| ) \big ) (1 - s^2)^{\frac{n - 3}{2}} \jj_n(r) r^{n - 1} \, ds \, dr \\
    & = (2 \pi)^{n/2} \int_0^\infty \biggl(\frac{1}{2^{n/2 - 1} \Gamma(\frac{n}{2})} - (r |\xi|)^{1 - n/2} J_{n/2 - 1}(r |\xi|) \biggr) r^{n - 1} \jj_n(r) \, dr ,
\end{split}\end{equation}
where $J_\nu$ is the Bessel function of the first kind; see~\cite[Chapter 10]{NIST:DLMF}. Hence, it follows that \(\Psi_n\) is rotationally invariant, that is \(\Psi_n(\xi)=\psi(\vert \xi \vert ) \), with $\psi$ given as the Hankel transform
\begin{align}
 \psi(\tau) = (2 \pi)^{n/2} \int_0^\infty \biggl(\frac{1}{2^{n/2 - 1} \Gamma(\frac{n}{2})} - (r \tau)^{1 - n/2} J_{n/2 - 1}(r \tau)\biggr) r^{n - 1} \jj_n(r) \, dr ,
\label{psidierre000}
\end{align} 
It is no mistake that $\psi$ has no subscript $n$: under our assumptions~\eqref{radialassumption} and~\eqref{nplus2} it turns out that the Fourier symbol $\Psi_{n + 2}$ of the operator $L_{n + 2}$ has the same profile function $\psi$. There are many different ways to prove this fact; for instance, it is a relatively simple consequence of Bochner's relation. A direct proof involving well-known properties of Bessel functions is perhaps even shorter: since, by~\eqref{nplus2}, $\jj_{n}(r) = 2 \pi \int_r^\infty t \,\jj_{n+2}(t) \, dt$, 
if we denote by~$\psi_n$ the quantity in~\eqref{psidierre000},
we have
\begin{align*}
 \psi_{n}(\tau) & = (2 \pi)^{n/2 + 1} \int_0^\infty \int_r^\infty \biggl(\frac{1}{2^{n/2 - 1} \Gamma(\frac{n}{2})} - (r \tau)^{1 - n/2} J_{n/2 - 1}(r \tau)\biggr) r^{n - 1} t \jj_{n + 2}(t) \, dt \, dr \\
 & = (2 \pi)^{n/2 + 1} \int_0^\infty \int_0^t \biggl(\frac{1}{2^{n/2 - 1} \Gamma(\frac{n}{2})} - (r \tau)^{1 - n/2} J_{n/2 - 1}(r \tau)\biggr) r^{n - 1} t \jj_{n + 2}(t) \, dr \, dt \\
 & = (2 \pi)^{n/2 + 1} \int_0^\infty  \biggl(\frac{1}{2^{n/2} \Gamma(\frac{n}{2} + 1)} - (t \tau)^{-n/2} J_{n/2}(t \tau)\biggr) t^{n + 1} \jj_{n + 2}(t) \, dt =\psi_{n+2}(\tau),
\end{align*}
where we have used Fubini's theorem and~\cite[Equation 10.22.1]{NIST:DLMF},
and this observation allows us to drop the index~$n$
and write simply~$\psi$ instead of~$\psi_n$.

Therefore, Bochner's relation and its consequences proved in Section~\ref{sec:bochner} apply to this setting. The previously established theory will give us information about \(L_n\) given information about \(L_{n+2}\), namely that if \(L_{n+2}\) admits the Harnack inequality then \(L_n\) admits the antisymmetric Harnack inequality.


\begin{corollary}
\label{cor:odd:flapBIS}
Let $f$ be a symmetric Schwartz function on $\R^n$. Let $\tilde{f}$ be the corresponding 3-isotropic Schwartz function on $\R^{n + 2}$, given by
\[
    \tilde{f}(\tilde{x}) = f\bigg (\sqrt{\tilde x_1^2 + \tilde x_2^2 + \tilde x_3^2}, \tilde x_4, \ldots, \tilde x_{n+2}\bigg ) .
\]
Finally, let $g(x) = x_1 f(x)$. Then
\[
    L_n g(x) = x_1 L_{n+2} \tilde{f}((x_1, 0, 0), x_2, \ldots, x_n) .
\]
\end{corollary}

\begin{proof} The proof is a consequence of Lemma~\ref{lem:odd:operator}.
Indeed, recall that the symbol~$\psi$ given in~\eqref{psidierre000}
is rotationally invariant, and operators $L_n$ and $L_{n + 2}$ satisfy formula~\eqref{FOU:SI}.
Accordingly, for every~$x=(x_1,x_2,\dots,x_n)\in\R^n$,
\begin{equation*}
L_ng(x)=L_n(x_1 f(x))=x_1L_{n+2}\tilde f((x_1,0,0),x_2,\dots,x_n),
\end{equation*}
as desired.
\end{proof}

Recall that $L_{n+2}$ maps $3$-isotropic functions to $3$-isotropic functions.
We now provide an extension of Corollary~\ref{cor:odd:flapBIS} to antisymmetric functions. This is the counterpart of Theorem~\ref{thm:odd:flap}
for general operators.

\begin{theorem}
\label{thm:odd:flapBIS}
Suppose that $\Omega$ is an open set of~$\R^n$. Let $u$ be an antisymmetric function on $\R^n$ such that $u\in C^2_{{\rm{loc}}}(\Omega)$, $x_1^{-1}u \in L^\infty(\R^n)$,  and let $\tilde{v}$ be the corresponding function on $\R^{n + 2}$ given by
\begin{equation}\label{IPBN}
 \tilde{v}(\tilde{x}) = \frac{u\big (\sqrt{\tilde x_1^2 + \tilde x_2^2 + \tilde x_3^2}, \tilde x_4, \ldots, \tilde x_{n+2} \big )}{\sqrt{\tilde{x}_1^2 + \tilde{x}_2^2 + \tilde{x}_3^2}} \, .
\end{equation}

Then:
\begin{enumerate}[label={\rm (\alph*)}]
\item\label{thm:odd:flap:aBIS} if $\Omega_0 = \{x \in \Omega : x_1 \ne 0\}$ and 
\[
 \tilde \Omega_0 = \bigg \{\tilde{x} \in \R^{n + 2} : \Big (\sqrt{\tilde x_1^2 + \tilde x_2^2 + \tilde x_3^2}, \tilde x_4, \ldots, \tilde x_{n+2} \Big ) \in \Omega_0\bigg \} ;
\]
then the function $\tilde{v}$ is $C^2_{{\rm{loc}}} (\tilde \Omega_0) \cap L^\infty(\R^{n+2}) $.

\item\label{thm:odd:flap:cBIS} for every $x \in \Omega_0$, we have that
\begin{equation}
\label{eq:odd:flapBIS}
L_nu(x) = x_1 L_{n+2} \tilde{v}((x_1, 0, 0), x_2, \ldots, x_n);
\end{equation}
\item\label{thm:odd:flap:dBIS} for an arbitrary $C_c^\infty$ function $g$ on $\Omega$, we have
\begin{equation}
\label{eq:odd:flap:weakBIS}
 \langle u, L_n g \rangle_{\R^n} = (2 \pi)^{-1} \langle \tilde{v}, L_{n+2} \tilde{f} \rangle_{\R^{n + 2}} ,
\end{equation}
where
\[ \tilde{f}(\tilde{x}) = \frac{g_A\big (\sqrt{\tilde x_1^2 + \tilde x_2^2 + \tilde x_3^2}, \tilde x_4, \ldots, \tilde x_{n+2} \big ) }{\sqrt{\tilde x_1^2 + \tilde x_2^2 + \tilde x_3^2}} \]
and
\[ g_A(x) = \frac 12 \big (  g(x_1,\dots,x_n) - g(-x_1,x_2,\dots,x_n) \big ) ; \]
\item\label{thm:odd:flap:eBIS} for an arbitrary $C_c^\infty$ function $\tilde{f}$ on
\[ \tilde \Omega=\bigg \{\tilde{x} \in \R^{n + 2} : \Big (\sqrt{\tilde x_1^2 + \tilde x_2^2 + \tilde x_3^2}, \tilde x_4, \ldots, \tilde x_{n+2} \Big ) \in \Omega \bigg \}, \]%
we have
\begin{equation}
\label{eq:odd:flap:reverseBIS}
\langle \tilde{v}, L_{n+2}\tilde{f} \rangle_{\R^{n + 2}} = 2 \pi \langle u, L_n g \rangle_{\R^n} ,
\end{equation}
where $g(x) = x_1 f(x)$ and $f$ is given in terms of the 3-isotropic symmetrization of $\tilde{f}$ by
\begin{equation*}
 f(x) = \frac{1}{4 \pi} \int_{\sph^2} \tilde{f}(|x_1| z, x_2, \ldots, x_n)\, \mathcal H^2(dz) ;
\end{equation*}
here $\sph^2$ is the unit sphere in $\R^3$ and $\mathcal H^2$ is the surface measure.
\end{enumerate}
\end{theorem}

\begin{proof}
Since the map $\tilde x \mapsto \big (\sqrt{\tilde x_1^2 + \tilde x_2^2 + \tilde x_3^2}, \tilde x_4, \ldots, \tilde x_{n+2} \big )$ is smooth from $\tilde \Omega_0$ to $\Omega_0$ and \(x \mapsto x_1^{-1} u(x)\) is in \(C^2_{{\rm{loc}}}(\Omega_0)\), it follows that \(v\in C^2_{{\rm{loc}}}(\tilde \Omega_0)\). Moreover, the assumption \(x_1^{-1}u \in L^\infty(\R^n)\) immediately implies that \(\tilde v \in L^\infty(\R^{n + 2})\), which gives~\ref{thm:odd:flap:aBIS}.

To prove the claim in~\ref{thm:odd:flap:cBIS}, we notice that if $f(x) = x_1^{-1} u(x)$ is a Schwartz function, item~\ref{thm:odd:flap:cBIS} reduces to Corollary~\ref{cor:odd:flapBIS}. The extension to general $u$ follows by mollification and cutting off. This is essentially the same as the argument applied in the proof of Theorem~\ref{thm:odd:flap}\ref{thm:odd:flap:c}, and so we omit the details.

The proofs of~\ref{thm:odd:flap:dBIS} and~\ref{thm:odd:flap:eBIS} are analogous to the proofs of Theorem~\ref{thm:odd:flap}~\ref{thm:odd:flap:d} and~\ref{thm:odd:flap:e} respectively.
\end{proof}

In preparation for the proof of Theorem~\ref{THM5.55}, we now give the definition of an operator satisfying the Harnack inequality.

\begin{definition} \label{mTbTJpHm}
Let \(\mathcal A\) be the collection of quadruples \(( \Omega,  K,  c,v)\) such that  \( \Omega \subset \R^n\) is open; $K \subset \Omega$ is compact and connected; \( c\in L^\infty(\Omega)\); and \(  v\in C^2_{{\rm{loc}}}( \Omega) \cap L^\infty  (\R^n)\) satisfies, in the weak\footnote{Saying that~\eqref{PDEBIS} is satisfied in the weak sense means here that
$$  \langle v,L_n\phi \rangle_{\R^n}+ \langle c   v,\phi \rangle_{\R^n}=0,$$
for every $C_c^\infty$ function $\phi$ on $\Omega$.

In this setting, results such as Theorem~\ref{thm:odd:flap}~\ref{thm:odd:flap:e}
and Theorem~\ref{thm:odd:flapBIS}~\ref{thm:odd:flap:eBIS} are helpful to check that
a given weak solution gives rise to a weak solution to the operator in higher dimension.} sense, \begin{align}
\label{PDEBIS}
\begin{PDE}
L_n  v + c   v &= 0, &\text{in } \Omega \\
 v &\geq 0, &\text{in } \R^n.
\end{PDE}
\end{align} We say that an operator \(L_n\) given by~\eqref{5XOlUbY7} admits the \emph{Harnack inequality} if, for all \((\Omega,K,c,v)\in \mathcal A\), there holds that \begin{align*}
\inf_ K  v \leq C \sup_K v
\end{align*} for some \(C=C(\Omega,K,\| c\|_{L^\infty(\Omega}) )>0\).
\end{definition}

\begin{remark}
\label{rem:harnack}
When $c = 0$ in~\eqref{PDEBIS}, that is, when $u$ is a harmonic function in $\Omega$ for the operator $L_n$, then Harnack inquality holds under fairly general conditions on the kernel $\jj_n$ of $L_n$ (see~\eqref{5XOlUbY7}). Indeed: Theorem~4.1 in~\cite{GK18} provides a Harnack inequality (in fact: a boundary Harnack inequality; see Remark~1.10(e) in that paper) when $K$ and $\Omega$ are balls, under the following assumptions: $\jj_n$ is positive and nonincreasing (so that the kernel of $L_n$ is isotropic and \emph{unimodal}), and there is a constant $C > 0$ such that $\jj_n(r + 1) \ge C \jj_n(r)$ for $r \ge 1$ (so that in particular $\jj_n(r)$ goes to zero as $r \to \infty$ no faster than some exponential).

By Theorem~1.9 in~\cite{GK18}, the constant in the Harnack inequality is \emph{scale-invariant} if $\jj_n$ is positive and decreasing and it satisfies the following scaling condition:
\[ \frac{\jj_n(r)}{\jj_n(R)} \le C \biggl(\frac{r}{R}\biggr)^{-n - \alpha} \]
for some constants $C > 0$ and $\alpha \in (0, 2)$.

While it is natural to expect that the results described above also hold for non-zero $c$, such extension does not seem to be available in literature. More restrictive results are given in \cite{KIM2022125746,MR3794384,MR4023466} and references therein.
\end{remark}

We have now developed the machinery to give the proof of Theorem~\ref{THM5.55}.

\begin{proof}[Proof of Theorem~\ref{THM5.55}]
From Theorem~\ref{thm:odd:flapBIS}~\ref{thm:odd:flap:cBIS}, we have that \begin{align*}
L_nu(x) &= x_1 L_{n+2}\tilde v((x_1, 0, 0), x_2, \ldots, x_n) 
\end{align*} for all \(x\in \Omega_0\); or equivalently, \begin{equation}\label{APIB}
\sqrt{\tilde x_1^2 + \tilde x_2^2 + \tilde x_3^2  } \, L_{n+2}\tilde v (\tilde x) = L_nu \Big (\sqrt{\tilde x_1^2 + \tilde x_2^2 + \tilde x_3^2}, \tilde x_4, \ldots, \tilde x_{n+2} \Big ) 
\end{equation} for all \(\tilde x \in \tilde \Omega_0\) (recall that $L_{n + 2} u$ is $3$-isotropic). It follows that if we denote $x = (\sqrt{\tilde x_1^2 + \tilde x_2^2 + \tilde x_3^2}, \tilde x_4, \ldots, \tilde x_{n+2})$, then \begin{align*}
\sqrt{\tilde x_1^2 + \tilde x_2^2 + \tilde x_3^2  } \, \big ( L_{n+2}\tilde v (\tilde x) + c(x) \tilde v (\tilde x) \big ) &=  L_nu(x)  + c(x) u(x)=0 .
\end{align*} Setting \(\tilde c (\tilde x) = c(x)\), we have that \begin{align*}
L_{n+2}\tilde v + \tilde c \tilde v = 0 \qquad \text{in } \tilde \Omega_0 . 
\end{align*} In fact, by Theorem~\ref{thm:odd:flapBIS}~\ref{thm:odd:flap:eBIS}, the above equality holds in the weak sense in $\tilde \Omega$. By assumption, \(L_{n+2}\) admits the Harnack inequality, so from the definition applied to the quadruple \((\tilde \Omega, \tilde K , \tilde c, \tilde v)\) with \(\tilde K = \big  \{\tilde x \in \R^{n + 2} : \big (\sqrt{\tilde x_1^2 + \tilde x_2^2 + \tilde x_3^2}, \tilde x_4, \ldots, \tilde x_{n + 2} \big ) \in K \big \}\), we have that \begin{align*}
\sup_{\tilde K } \tilde v \leq C(\tilde \Omega,\tilde K,\|\tilde c\|_{L^\infty(\tilde \Omega)}) \inf_{\tilde K} \tilde v .
\end{align*} We conclude that \begin{align*}
\sup_{ x\in K_0} \frac{u(x)}{x_1} \leq C(\tilde \Omega,\tilde K,\|c\|_{L^\infty(\Omega_+)}) \inf_{x\in K_0}  \frac{u(x) }{x_1} ,
\end{align*} as desired. 
\end{proof}

\section{Additional results: a boundary weak Harnack
inequality for super-solutions and a boundary local boundedness for sub-solutions}\label{sodjvlfORTYHSJDgh}

We recall that in~\cite{dtv}
the antisymmetric boundary
Harnack inequality was obtained as a byproduct of a weak Harnack inequality for super-solutions and a local boundedness for sub-solutions. These results, which are in fact of independent interest,
may also be obtained by the methods showcased in this paper, as sketched here below
(and for this we will use the norm notation in~\eqref{AAAnorm}
and the setting in~\eqref{IPBN}).

Following is the antisymmetric boundary weak Harnack inequality.

\begin{theorem} 
Let \(M\in \R\), \(\rho>0\), and \(c\in L^\infty(B_\rho^+)\). Suppose that \(u\in C^{2s+\alpha}(B_\rho)\cap \mathscr{A}_s(\R^n)\) for some \(\alpha > 0\) with \(2s+\alpha\) not an integer, \(u\) is non-negative in \(\R^n_+\) and satisfies \begin{align*}
(-\Delta)^s u +cu \geqslant -Mx_1 \qquad \text{in } B_\rho^+.
\end{align*}

Then there exists \(C_\rho>0\) depending only on \(n\), \(s\), \(\| c \|_{L^\infty(B_\rho^+)}\), and \(\rho\) such that \begin{align*}
 \| u\|_{  \mathscr A_s(\R^n) }  &\leqslant  C_\rho \left(  \inf_{B_{\rho/2}^+} \frac{u(x)}{x_1} +M \right).
\end{align*}  
\end{theorem}

\begin{proof} Using~\eqref{APIB} and the notation~$x = (\sqrt{\tilde x_1^2 + \tilde x_2^2 + \tilde x_3^2}, \tilde x_4, \ldots, \tilde x_{n+2})$ we have that, in the $(n+2)$-dimensional ball~${\mathcal{B}}_\rho$ of radius~$\rho$,
\begin{equation}\label{apdoghAi-2083c4} \sqrt{\tilde x_1^2 + \tilde x_2^2 + \tilde x_3^2  } \, (-\Delta)^s\tilde v (\tilde x) = (-\Delta)^su \Big (\sqrt{\tilde x_1^2 + \tilde x_2^2 + \tilde x_3^2}, \tilde x_4, \ldots, \tilde x_{n+2} \Big ) 
\end{equation}
and thus
\begin{align*}
(-\Delta)^s \tilde v +c\tilde v \geqslant -M .
\end{align*}

Then, by the interior  weak Harnack inequality (see \cite[Theorem 2.2]{MR4023466}),
$$\|\tilde v\|_{  \mathscr L_s(\R^{n+2}) } \leqslant C_\rho \left( \inf_{{\mathcal{B}}_{\rho/2}} \tilde v + M \right) .$$
Also, since the $n$-dimensional ball~$B_{\rho/2}$ is contained in~${\mathcal{B}}_{\rho/2}$, we have that
$$  \inf_{{\mathcal{B}}_{\rho/2}} \tilde v\le\inf_{B_{\rho/2}} \tilde v=\inf_{B_{\rho/2}^+} \frac{u(x)}{x_1} .
$$
These observations and the equivalence of the norms in~\eqref{AAAnorm} yield the desired result.
\end{proof}

Here below is the local boundedness for antisymmetric sub-solutions.

\begin{theorem} \label{R5gjD59t}
Let \(M \geqslant 0\), \(\rho\in(0,1)\), and \(c\in L^\infty(B_\rho^+)\). Suppose that \(u \in C^{2s+\alpha}(B_\rho)\cap \mathscr A_s (\R^n)\) for some \(\alpha>0\) with \(2s+\alpha\) not an integer, and \(u\) satisfies  \begin{align*}
(-\Delta)^su +cu &\leqslant M x_1 \qquad \text{in } B_\rho^+.
\end{align*}

Then there exists \(C_\rho>0\) depending only on \(n\), \(s\), \(\| c \|_{L^\infty(B_\rho^+)}\), and \(\rho\) such that  \begin{align*}
\sup_{x\in B_{\rho/2}^+} \frac{ u(x)}{x_1} &\leqslant C_\rho (  \| u\|_{  \mathscr A_s(\R^n) } +M  ) .
\end{align*} 
\end{theorem}

Before we give the proof of Theorem~\ref{R5gjD59t}, we stress 
that it requires local boundedness for sub-solutions in a ball (with no anti-symmetry assumption). If there are no zero-order terms then this is well-known, see for example \cite[Theorem~5.1]{CS11}. When there are zero-order terms, this is still likely to be well-known to experts, but references are much harder to find. In \cite{kim_holder_2023}, a local boundedness result for sub-solutions with zero-order terms is proved, but it takes a different form to what we would like to use in the proof of Theorem~\ref{R5gjD59t}
(on a similar note, the norm used in~\cite[Theorem~6.2]{MR2825646}
would be insufficient for the proof of Theorem~\ref{R5gjD59t}). Thus, for the sake of clarity, we present below a proof of how to go from the result in \cite{kim_holder_2023} to the exact formulation that we use in Theorem~\ref{R5gjD59t}.

\begin{proposition} \label{RmPesNtb}
    Let \(M > 0\), \(\rho>0\), and \(c \in  L^\infty(B_\rho)\). Suppose that \(u:\R^n\to\R\) with \(u \in C^{2s+\alpha}(B_\rho)\) for some \(\alpha > 0\) with \(2s + \alpha\) not an integer, satisfies \begin{align*}
    \int_{\R^n} \frac{\vert u(x) \vert }{(1+\vert x \vert)^{n+2s}} \, dx<+\infty,
\end{align*} and satisfies \begin{align*}
     (-\Delta)^su + cu \leq  M \qquad \text{in }B_\rho.
\end{align*} Then \begin{align*}
    \sup_{B_{\rho/2}} u \leq C \bigg ( M+ \int_{\R^n} \frac{\vert u(x) \vert }{(1+\vert x \vert)^{n+2s}}  \bigg).
\end{align*} The constant \(C>0\) depends only on \(n\), \(s\), \(\rho\), \(\| c\|_{L^\infty(B_\rho)}\). 
\end{proposition}

\begin{proof}
   
  Let \(v=u- \tilde M \chi_{B_{4\rho}\setminus B_{2\rho}}\) with \(\tilde M\) a constant to be chosen later. Then, we have that \begin{align*}
        (-\Delta)^sv+cv&\leq M-\tilde M (-\Delta)^s\chi_{B_{4\rho}\setminus B_{2\rho}} \qquad \text{in }B_\rho  
    \end{align*} Furthermore, for \(x\in B_\rho\) \begin{align*}
        (-\Delta)^s \chi_{B_{4\rho}\setminus B_{2\rho}}(x) &= -C\int_{B_{4\rho}\setminus B_{2\rho}} \frac{d y }{\vert x - y \vert^{n+2s}} \geq  - C \rho^{-2s}. 
    \end{align*} Hence, \begin{align*}
          (-\Delta)^sv+cv&\leq M-C \tilde M\rho^{-2s} = 0 \quad \text{in }B_\rho
    \end{align*} with \(\tilde M = C^{-1} M\rho^{2s}\).

    Next, we use \cite[Theorem 1.1]{kim_holder_2023} with \(v\) to obtain \begin{align*}
        \sup_{B_{\rho /2}} v \leq C \bigg ( \rho^{2s} \int_{\R^n\setminus B_{\rho/2}} \frac{\vert v(y)\vert}{\vert y \vert^{n+2s}} \, dy + \rho^{-n} \| v\|_{L^2(B_{\rho})} \bigg ).   
    \end{align*} Fix \(\varepsilon>0\). By Young's inequality, \begin{align*}
        \bigg ( \int_{B_\rho} v^2 \, dx \bigg )^{\frac 12}  \leq (\sup_{B_\rho} v )^{\frac12} \bigg (\int_{B_\rho} \vert v\vert  \, dx \bigg )^{\frac12} \leq \varepsilon( \sup_{B_\rho} u )+ \frac C{\varepsilon} \int_{B_\rho} \vert v \vert \, dx , 
    \end{align*} so \begin{align*}
        \sup_{B_{\rho /2}} v &\leq C_\rho \bigg (  \varepsilon  ( \sup_{B_\rho} v ) +  \int_{\R^n\setminus B_{\rho/2}} \frac{\vert v(y)\vert}{\vert y \vert^{n+2s}} \, dy+ \frac 1{\varepsilon} \int_{B_\rho} \vert v \vert \, dx\bigg ) \\
        &\leq C_\rho \bigg (  \varepsilon  ( \sup_{B_\rho} v ) + C_\varepsilon \int_{\R^n} \frac{\vert v(x) \vert }{(1+\vert x \vert)^{n+2s}}  \bigg ). 
    \end{align*} Then \cite[Lemma B.1]{cabre_bernstein_2022} implies \begin{align*}
        \sup_{B_{\rho /2}} v &\leq C_\rho \int_{\R^n} \frac{\vert v(x) \vert }{(1+\vert x \vert)^{n+2s}} .
    \end{align*} Thus, \begin{align*}
         \sup_{B_{\rho /2}} u&= \sup_{B_{\rho /2}} v\\
         &\leq C_\rho \int_{\R^n} \frac{\vert v(x) \vert }{(1+\vert x \vert)^{n+2s}} \\
         &\leq C_\rho \bigg (\int_{\R^n} \frac{\vert u(x) \vert }{(1+\vert x \vert)^{n+2s}} + \tilde M \int_{B_{4\rho}\setminus B_{2\rho}} \frac1{(1+\vert x \vert)^{n+2s}} \bigg )\\
         &\leq C_\rho \bigg ( M +\int_{\R^n} \frac{\vert u(x) \vert }{(1+\vert x \vert)^{n+2s}} \bigg ) .
    \end{align*}
\end{proof}

We now give the proof of Theorem~\ref{R5gjD59t}.

\begin{proof}[Proof of Theorem~\ref{R5gjD59t}] It follows from~\eqref{apdoghAi-2083c4} that~$
(-\Delta)^s\tilde v +c\tilde v \leqslant M $ in the $(n+2)$-dimensional ball~${\mathcal{B}}_\rho$ of radius~$\rho$. Then one can use the Proposition~\ref{RmPesNtb} to deduce that
$$ \sup_{{\mathcal{B}}_{\rho/2}} \tilde v \leqslant C_\rho \left(\|\tilde v\|_{  \mathscr L_s(\R^{n+2}) }+ M \right) $$and the desired result follows as above.
\end{proof}

\section*{Data Availability Statement}
This is a theoretical mathematics paper concerning partial differential equations and, as such, contains no data. We provide this statement only to satisfy the publisher's formal requirements, though we note that, in this context, 
such a statement is obviously redundant.

\section*{Acknowledgments}
We thank the Referees for their useful and interesting comments.

SD has been supported by the Australian Future Fellowship
FT230100333 ``New perspectives on nonlocal equations''.
JT is supported by an Australian Government Research Training Program Scholarship.
EV has been supported by the Australian Laureate Fellowship FL190100081 ``Minimal surfaces, free boundaries and partial differential equations.''
MK has been supported by National Science Centre, Poland, 2023/49/B/ST1/04303.

This research was funded in whole or in part by National Science Centre, Poland, 2023/49/B/ST1/04303. For the purpose of Open Access, the authors have applied a CC-BY public copyright licence to any Author Accepted Manuscript (AAM) version arising from this submission.

\appendix

\section{Approximation of functions in a class of weighted $L^1$ spaces} \label{ZVk7DrGw}
In this section, we give a result that is possibly well-known to experts, but 
perhaps not easy to find explicitly stated and proved in the literature (for example, 
a similar result was claimed to be true in~\cite{SILTHES}
and an explicit proof was recently written in~\cite[page~237, Chapter~12]{STINGABOOK}).
We provide full details for the facility of the reader.

\begin{proposition}
  Let $\eta_\varepsilon$ be a standard mollifier on $\R^n$, that is $\eta_1$ is a smooth, nonnegative function with integral $1$ and support contained in the unit ball, and $\eta_\varepsilon(x) = \varepsilon^{-n } \eta_1(x / \varepsilon)$. Moreover, let \(\alpha>0\) and \(\mu\) be the positive measure on \(\R^n\) given by \begin{align*}
     d \mu = \frac{d x }{(1+\vert x \vert)^\alpha}. 
  \end{align*}If \(u\in  L^1(\mu)\) and \( u ^\varepsilon = u \ast \eta_\varepsilon\) then \begin{align*}
       \| u^\varepsilon -u  \|_{ L^1(\mu)} \to 0
  \end{align*} as \(\varepsilon \to 0^+\).
\end{proposition}

\begin{proof}
    First, suppose that \(u\) is a continuous, compactly supported function. Then \begin{align*}
      \| u^\varepsilon-u\|_{L^1(\mu)} \leq \int_{\R^n} \int_{\R^n} \frac{\vert u(y)-u(x) \vert \eta_\varepsilon (y-x)}{(1+\vert x \vert)^\alpha}  \,d y \,d x. 
    \end{align*} Choosing \(R>1\) so that \(u(x)=0\) when \(\vert x \vert \geq R\), we have \begin{align*}
       \| u^\varepsilon-u\|_{L^1(\mu)} &\leq \int_{\R^n} \int_{B_{\varepsilon}(x)} \frac{\vert u(y)-u(x) \vert \eta_\varepsilon (y-x)}{(1+\vert x \vert)^\alpha}  \,d y \,d x\\
       &= \int_{B_{2R}} \int_{B_{\varepsilon}(x)} \frac{\vert u(y)-u(x) \vert \eta_\varepsilon (y-x)}{(1+\vert x \vert)^\alpha}  \,d y \,d x,
    \end{align*} where to obtain the second line above we use that if \(\vert x \vert \geq 2R\) and \(\vert x-y\vert < \varepsilon\) then \(\vert y \vert \geq 2 R - \varepsilon > R\) for \(\varepsilon\) small, so \(u(x)=u(y)=0\). 
    
    Next, if \(\omega_u\) denotes the modulus of continuity of \(u\), we have \begin{align*}
          \| u^\varepsilon-u\|_{L^1(\mu)} &\leq \omega_u(\varepsilon) \int_{B_{2R}} \int_{B_{\varepsilon}(x)} \frac{\eta_\varepsilon (y-x)}{(1+\vert x \vert)^\alpha}  \,d y \,d x \\
          &\leq  \omega_u(\varepsilon) \int_{B_{2R}}  \frac{d x }{(1+\vert x \vert)^\alpha}\to 0
    \end{align*} as \(\varepsilon \to 0^+\).

For a general \(u\in L^1(\mu)\), using e.g.~\cite[Theorem~3.14]{MR924157}
we may approximate \(u\) by a continuous, compactly supported function \(u_k\) satisfying \begin{align*}
    \| u- u_k \|_{L^1(\mu)} \leq \frac 1 k .
\end{align*} Then \begin{align*}
     \| u^\varepsilon- (u_k)^\varepsilon \|_{L^1(\mu)}&\leq \int_{\R^n} \int_{B_\varepsilon(x)} \frac{\vert u(x)-u_k(y)\vert \eta_\varepsilon(x-y) }{(1+\vert x \vert)^\alpha} \, dy \, dx .
\end{align*} Since \(1+\vert y \vert \leq (1+\varepsilon) (1+\vert x \vert) \) whenever \(\vert x-y\vert <\varepsilon\), it follows that \begin{align*}
    \| u^\varepsilon- (u_k)^\varepsilon \|_{L^1(\mu)}&\leq (1+\varepsilon)^\alpha \int_{\R^n} \int_{B_\varepsilon(x)} \frac{\vert u(x)-u_k(y)\vert \eta_\varepsilon(x-y) }{(1+\vert y \vert)^\alpha} \, dy \, dx \\
    &= (1+\varepsilon)^\alpha \int_{\R^n} \frac{\vert u(x)-u_k(y)\vert  }{(1+\vert y \vert)^\alpha} \, dy \\
    &\leq \frac{(1+\varepsilon)^\alpha } k . 
\end{align*} Since \begin{align*}
    \| u^\varepsilon - u \|_{L^1(\mu)} &\leq \| u^\varepsilon - (u_k)^\varepsilon \|_{L^1(\mu)}+\| (u_k)^\varepsilon - u_k \|_{L^1(\mu)}+\| u_k - u \|_{L^1(\mu)}\\ &\leq \frac{(1+\varepsilon)^\alpha } k  + \| (u_k)^\varepsilon - u_k \|_{L^1(\mu)}+\frac 1 k 
\end{align*} and \begin{align*}
    \limsup_{\varepsilon \to 0^+ } \vert (u_k)^\varepsilon(x) - u_k(x)\vert = 0,  
\end{align*}we conclude that \begin{align*}
     \limsup_{\varepsilon \to 0^+ } \| u^\varepsilon - u \|_{L^1(\mu)} \leq \frac 2 k 
\end{align*} for all \(k\geq 1\), so \(\displaystyle\lim_{\varepsilon \to 0^+ } \| u^\varepsilon - u \|_{L^1(\mu)}\) exists and equals zero.
\end{proof}

\vfill

\end{document}